\documentclass[ECP]{ejpecp} % replace ECP by EJP if needed.
\usepackage{tikz}
\usepackage{mhequ}
\usetikzlibrary{arrows, shapes, snakes}
 \usetikzlibrary {arrows.meta} 

\SHORTTITLE{KPZ with rougher than white noise}

\TITLE{Weak coupling limit of KPZ with rougher than white noise}
\AUTHORS{%
  M\'at\'e~Gerencs\'er\footnote{TU Wien, Austria.
    \EMAIL{mate.gerencser@tuwien.ac.at}}\orcid{0000-0002-7276-7054}
  \and %% remove this line and below if single author
  Fabio~Toninelli\footnote{TU Wien, Austria. \EMAIL{fabio.toninelli@tuwien.ac.at}} \orcid{0000-0003-1710-8811}}%AUTHORS
 % Separate items with ;

\AMSSUBJ{60H15; 60L30} % Edit. Separate items with ;
\SUBMITTED{June 12, 2024} % Edit.
\ACCEPTED{March 11, 2025} % Edit.

\VOLUME{30}
\YEAR{2025}
\PAPERNUM{34}
\DOI{10.1214/25-ECP675}

\ABSTRACT{We consider the KPZ equation in $1$ spatial dimension with noise that is rougher than white by an exponent $\gamma>1/4$. Under a weak coupling limit, formally removing the nonlinearity from the equation, we show using regularity structures that the renormalised solutions converge to a Gaussian limit that is different from the solution of the linear part of the equation. The regime of this effect has a nontrivial overlap with the subcritical regime $\gamma<1/2$.}

\newcommand{\eps}{\varepsilon}

\tikzset{ 
noise/.style={very thin,shape=rectangle,fill=white,draw=black,inner sep=0pt,minimum size=1.1mm},
intF/.style={very thin,shape=rectangle,fill=black,draw=black,inner sep=0pt,minimum size=1.1mm},
int/.style={very thin,shape=circle,fill=black,draw=black,inner sep=0pt,minimum size=1.2mm},
nod/.style={very thin,shape=circle,fill=black!25!white, draw=black!25!white,inner sep=0pt,minimum size=0.9mm},
HK/.style={draw=black, arrows = {-Stealth[length=6pt, inset=4pt]}, shorten >=0.08cm},
test/.style={line width=0.55mm, draw=black!25!white, arrows = {-Stealth[length=6pt, width=6pt, inset=3pt]}, shorten >=0.05cm},
}

\newcommand{\parr}{\mathop{\mathrm{par}}}
\newcommand{\new}{\mathop{\mathrm{new}}}
\newcommand{\R}{\mathbb{R}}

\newcommand{\Z}{\mathbb{Z}}
\newcommand{\cP}{\mathcal{P}}
\newcommand{\cC}{\mathcal{C}}
\newcommand{\E}{\mathbf{E}}
\newcommand{\T}{\mathbb{T}}
\newcommand{\dd}{\partial}
\newcommand{\one}{\mathbf{1}}
\newcommand{\cD}{\mathcal{D}}
\newcommand{\cI}{\mathcal{I}}
\newcommand{\cR}{\mathcal{R}}
\newcommand{\cT}{\mathcal{T}}
\newcommand{\PPi}{\mathbf{\Pi}}

\begin{document}

%%%%%%%%%%%%%%%%%%%%%%%%%%%%%%%%%%%%%%%%%%%%%%%%%%%%%%%%%%%%%%%%%%%
%%                                                               %%
%% No need for \maketitle.                                       %%
%%                                                               %%
%%%%%%%%%%%%%%%%%%%%%%%%%%%%%%%%%%%%%%%%%%%%%%%%%%%%%%%%%%%%%%%%%%%

%%%%%%%%%%%%%%%%%%%%%%%%%%%%%%%%%%%%%%%%%%%%%%%%%%%%%%%%%%%%%%%%%%%
%%                                                               %%
%% Please replace what follows by the body of your article       %%
%% (up to the bibliography):                                     %%
%%                                                               %%
%%%%%%%%%%%%%%%%%%%%%%%%%%%%%%%%%%%%%%%%%%%%%%%%%%%%%%%%%%%%%%%%%%%

\section{Introduction}

In this note we look at the $1+1$-dimensional KPZ equation on the torus $\mathbb T=\mathbb R/\mathbb Z$, but with rougher than space-time white noise:
\begin{equ}\label{eq:KPZ}
(\dd_t -\Delta) h= (\dd_x h)^2 + D^\gamma\xi,\qquad h(0,x)=\psi(x),
\end{equ}
where $\gamma\geq 0$ and $D=(2\pi)^{-1}(-\Delta)^{1/2}$, that is, $D$ is the linear operator that acts in Fourier space as multiplication by $|k|$.
When $\gamma=0$, this is the classical KPZ equation, solved in a very robust sense by Hairer \cite{HairerKPZ}, by what has since became the celebrated theory of regularity structures \cite{H0}.
For $\gamma<1/4$, \eqref{eq:KPZ} can still be solved by regularity structures \cite{Hoshino}.
However, the exponent suggested by scaling is different: the equation is \emph{subcritical} for $\gamma<1/2$. Indeed, formally, if $h$ solves \eqref{eq:KPZ} then using the scaling properties of $D$ and $\xi$, we have that for any $\lambda>0$, $h^\lambda$ defined by $h^\lambda(t,x)=\lambda^{-1/2+\gamma}h(\lambda^2t,\lambda x)$ solves
\begin{equ}
(\dd_t -\Delta) h^\lambda= \lambda^{1/2-\gamma}(\dd_x h^\lambda)^2 + D^\gamma\tilde\xi
\end{equ}
with another space-time white noise $\tilde\xi$. When $\gamma<1/2$, on small scales (that is, $\lambda\to 0$) the nonlinearity vanishes, which is the informal notion of scaling subcriticality.

From the point of view of the theory of regularity structures of \cite{H0}, the threshold $1/4$ is where the so-called \emph{super-regularity} condition \cite[Def.~2.28]{CH} fails.
The main difficulty is that in the first steps of Picard's iteration stochastic objects with infinite variance arise, which can not be treated with the usual renormalisation that removes infinite expectations. 
The same threshold appears and have different interpretation in the It\^o solution theory of the stochastic heat equation driven by rough noise, formally obtained by applying the Cole-Hopf transform $u=e^h$ to \eqref{eq:KPZ} \cite{roughSHE, Mate}.
We are interested in the regime $\gamma>1/4$, where the equation may even be subcritical (if $\gamma\in(1/4,1/2)$), but is beyond the scope of the typical local solution theory provided by regularity structures.

As is already the case for the classical KPZ equation, \eqref{eq:KPZ} is ill-posed as written: since even the solution to the linear equation is not differentiable, one can only expect $\partial_x h$ to exist in a distributional sense, and distributions cannot be squared.
This singular product has to be renormalised by an infinite vertical shift of the interface, which one can formulate via approximations as follows.
Let $\rho:\R\to\R$ be a symmetric, nonnegative, smooth function supported on the ball of radius $1/4$ and integrating to $1$. Define $\rho^\eps(r)=\eps^{-1}\rho(\eps^{-1}r)$, $\varrho^\eps(t,x)=\rho^{\eps^2}(t)\rho^\eps(x)$, and $\xi^\eps=\varrho^\eps\ast\xi$.
The domain of the spatial convolution is understood to be the torus $\T=\R/\Z$, as well as any other spatial domain in the sequel, unless otherwise noted.
Then  \cite{HairerKPZ, Hoshino} considers the equation
\begin{equ}\label{eq:KPZ-class-eps}
(\dd_t -\Delta) h^\eps= (\dd_x h^\eps)^2 -C^\eps + D^\gamma\xi^\eps,\qquad h^\eps(0,x)=\psi(x),\quad x\in\mathbb T
\end{equ}
for $\gamma<1/4$ and $\psi\in \cC^{\theta}(\T)$ for some $\theta>0$, where $C^\eps$ is a some suitable $\eps$-dependent constant (counter-term), 
and show the convergence of the solutions $h^\eps$ as $\eps\to0$ to a nontrivial limit $h$, independent of $\rho$.
{\color{black}Here and below, $\cC^\alpha$ stands for the inhomogeneous H\"older-Besov space $B^\alpha_{\infty,\infty}$, see e.g. \cite[Sec.~3-4]{MW}. It is known that for $\alpha\in(0,1)$ these spaces coincide with the usual spaces of $\alpha$-H\"older continuous functions, which are straightforward to interpret also for functions taking values in any metric space.}

In this work, in the regime $\gamma>1/4$ we consider a different
approximating equation:
\begin{equ}\label{eq:KPZ-eps}
(\dd_t -\Delta) h^\eps= \eps^{2\gamma-1/2}(\dd_x h^\eps)^2 -C^\eps + D^\gamma\xi^\eps,\qquad h^\eps(0,x)=\zeta^\eps(x)+\psi(x).
\end{equ}
The term $\zeta^\eps$, that modifies the initial condition, is defined and explained below.
Note that since $\gamma>1/4$, the factor $ \eps^{2\gamma-1/2}$ formally removes the nonlinearity from the equation in the $\eps\to 0$ limit. For most subcritical equations
%(including \eqref{eq:KPZ} $\gamma<1/4$),
this would simply result in $h^\eps$ converging to the solution of the linear equation. Our result is that for $\gamma>1/4$, $h^\eps$ converges to the solution of a \emph{different} linear equation.

The function $\zeta^\eps$ that appears in the initial condition is defined as
\begin{equ}
\zeta^\eps(x)=\int_{(-\infty,0]\times\T} \cP(-s,x-y)D^\gamma\xi^\eps(s,y)\,ds\,dy,
\end{equ}
where $\cP$ denotes the heat kernel.
It is standard that for any $\kappa>0$,  $\zeta^\eps$ converges in $\cC^{1/2-\gamma-\kappa}(\T)$ in probability to a Gaussian random {\color{black}field} $\zeta^0$.
It is also easy to check that the law of $\zeta^\eps$ is stationary for the linear equation
\begin{equation}
(\dd_t -\Delta) X^\eps=  D^\gamma\xi^\eps,
\end{equation} 
obtained by dropping the non-linearity and the counterterm.
This stationarity is convenient for the analysis, but we do not believe it to be essential, see Remark \ref{rem:ini} below.
%For example, taking the original initial condition $\psi$ would result in a limiting noise $\bar \xi$ that is nonstationary. 
On the other hand, we do not expect the non-linear equation to have an explicit stationary measure, or that it will be Gaussian.

Note furthermore that for each $\eps>0$ the solution $h^\eps$ to
\eqref{eq:KPZ-eps} exists for all times. Indeed, by the usual Cole-Hopf transform, $\exp(\eps^{2\gamma-1/2} h^\eps)$ solves a linear equation with smooth coefficients, hence exists for all times and stays strictly positive.

\begin{theorem}\label{thm:main}
Let $\kappa,T>0$, $\theta>0$, and let $\psi\in \cC^{\theta}(\T)$. For any $\gamma>1/4$ there exists a nonzero constant $c_{\gamma,\rho}$ and a choice of constants $C^\eps$ such that $h_\eps\to\tilde h$ in law in $C([0,T];\cC^{\theta\wedge(1/2-\gamma)-\kappa}(\T))$ as $\eps\to 0$, where $\tilde h$ is the Gaussian process characterised by the linear equation
\begin{equ}
(\partial_t-\Delta)\tilde h=c_{\gamma,\rho}\tilde\xi+D^\gamma\xi,\qquad \tilde h(0,x)=\zeta^0(x)+\psi(x),
\end{equ}
where $\tilde \xi$ is a space-time white noise independent of $\xi$.
\end{theorem}

\begin{remark}\label{rem:Martin}
Another natural choice to tame the exploding variance is to put a vanishing constant not in front of the nonlinearity, but rather in front of the noise. This very question was addressed by \cite{HairerVar} during the preparation of the present work, showing (in the case $\gamma=1$) that with the correct prefactor in front of the noise{\color{black},} the solutions converge in law to the solution of the classical (!) KPZ equation driven by the same noise $\tilde \xi$ as above.
\end{remark}

Theorem \ref{thm:main} will follow from Theorem \ref{thm:main2} below.
For the setup of the latter, we introduce some notation.
Set $\beta_\gamma=2\gamma-1/2$.
Take a constant ${\color{black}C_1^\eps}$ and define the processes $X^\eps$ and $Y^\eps$ as
\begin{equs}
X^\eps(t,x)&=\int_{(-\infty,t]\times\T}\cP(t-s,x-y)D^\gamma\xi^\eps(s,y)\,ds\,dy,
\\
Y^\eps(t,x)&=\int_{[0,t]\times\T}\cP(t-s,x-y)\eps^{\beta_\gamma}\big((\dd_x X^\eps(s,y))^2-{\color{black}C^\eps_1}\big)\,ds\,dy.
\end{equs}
Equivalently, they are solutions of the equations
\begin{equs}[eq:system0]
(\dd_t -\Delta)X^\eps &= D^\gamma\xi^\eps,\qquad &X^\eps(0,x)&=\zeta^\eps(x),
\\
(\dd_t -\Delta)Y^\eps &= \eps^{\beta_\gamma}\big((\dd_x X^\eps)^2-{\color{black}C^\eps_1}\big),\qquad &Y^\eps(0,x)&=0,
\end{equs}
respectively.
Then with $ {\color{black}C^\eps_2=C^\eps-\eps^{\beta_\gamma}C^\eps_1}$, the remainder $Z^\eps=h^\eps-X^\eps-Y^\eps$ satisfies
\begin{equs}[eq:Z]
(\dd_t-\Delta)Z^\eps&=\eps^{\beta_\gamma}(\dd_x Z^\eps)^2+2\eps^{\beta_\gamma}(\dd_x Z^\eps)(\partial_x X^\eps)+2\eps^{\beta_\gamma}(\dd_x Z^\eps)(\dd_x Y^\eps)
\\
&\qquad+2\eps^{\beta_\gamma}(\dd_x X^\eps)(\dd_x Y^\eps)+\eps^{\beta_\gamma}(\dd_x Y^\eps)^2
-{\color{black}C^\eps_2},\qquad Z^\eps(0,x)=\psi(x).
\end{equs}
Note that $X^\eps$ is also meaningful for $\eps=0$: indeed, writing equivalently
\begin{equ}
X^\eps(t,x)=D^\gamma\bar X^{\eps}(t,x):=D^\gamma\int_{(-\infty,t]\times\T}(\varrho^\eps\ast\cP_0)(t-s,x-y)\xi(ds,dy)
\end{equ}
(here $\cP_0$ is the projection of $\cP$ to the spatially zero mean functions, note that $D^\gamma\cP_0= D^\gamma \cP$),
the natural limit of $X^\eps$ is $X^0=D^\gamma\bar X^0$, where
\begin{equ}
      \bar X^0(t,x)=\int_{(-\infty,t]\times\T}\cP_0(t-s,x-y)\xi(ds,dy).
\end{equ}
As for $Y^0$, we instead define it as the solution of 
\begin{equ}
Y^0(t,x)=c_{\gamma,\rho}\int_{[0,t]\times\T}\cP(t-s,x-y)\tilde \xi(s,y)\,ds\,dy,
\end{equ}
where $\tilde\xi$ is a space-time white noise independent of $\xi$ and $c_{\gamma,\rho}$ is a constant given by \eqref{eq:beautiful-c} below.

The following then clearly implies Theorem \ref{thm:main}.
\begin{theorem}\label{thm:main2}
Let $\gamma>1/4$, $T>0$, $0<\theta<\theta'<1/4$, $\alpha\in(0,1/2)$.
\begin{enumerate}
\item[(a)] Suppose that $\eps^{\beta_\gamma}\big(\E(\dd_x X^\eps)^2(0,0)-{\color{black}C^\eps_1}\big)\to0$. Then $(X^\eps,Y^\eps)\to(X^0,Y^0)$ as $\eps\to 0$ in law in $C([0,T]; \cC^{\alpha-\gamma}(\T))\times C([0,T];\cC^{\alpha}(\T))$.

\item[(b)] Let $\psi\in\cC^{\theta'}$.
 Then there exists a choice of constants ${\color{black}C_1^\eps}$ satisfying the condition of (a) and a choice of constants ${\color{black}C^\eps_2}$ such that $Z^\eps\to \bar Z$ as $\eps\to 0$ in probability
in $C([0,T];\cC^\theta(\T))$, where $\bar Z$ is the solution of the heat equation with initial condition $\psi$.
\end{enumerate}
\end{theorem}

\section{Proof of Theorem \ref{thm:main2}}
\subsection{Proof of (a)}\label{sec:proof(a)}

Without loss of generality one may assume $\E(\dd_x X^\eps)^2(0,0)={\color{black}C^\eps_1}$. Also, it suffices to show the convergence $(\bar X^\eps,Y^\eps)\to(\bar X^0,Y^0)$ in law in $\big(C([0,T],\cC^{\alpha})\big)^2$.
To do so, there are two tasks to be done. First, one needs to show tightness of $(\bar X^\eps,Y^\eps)$ in $\big(C([0,T],\cC^{\alpha})\big)^2$. Second, one needs to identify the limit. This is actually more convenient to do on the level of the right-hand sides of the equations that $\bar X^\eps$ and $Y^\eps$ solve: 
defining
\begin{equ}
(\xi^\eps,\tilde \xi^\eps):=\Big(\xi^\eps,\eps^{\beta_\gamma}\big((\partial_x X^\eps)^2-\E(\partial_x X^\eps)^2\big)\Big),
\end{equ}
it suffices to show that any finite dimensional marginal of any subsequential limit of $(\xi^\eps,\tilde \xi^\eps)$ coincides with that of $(\xi,\tilde \xi)$.

We start with the latter task.
Since the claimed limit is Gaussian, it suffices to show that for any two test functions $\psi,\varphi\in C_c^\infty(\R\times\T)$ one has
\begin{equ}
\Phi^\eps:=\xi^\eps(\psi)+\tilde\xi^\eps(\varphi)\,\,\overset{\mathrm{law}}{\to}\,\,\xi(\psi)+\tilde\xi(\varphi)=:\Phi\sim N\big(0,\|\psi\|_{L^2}^2+c_{\gamma,\rho}^2\|\varphi\|_{L^2}^2\big),
\end{equ}
where $c_{\gamma,\rho}>0$ is to be determined.
Consider the Fourier transform in the spatial variable, henceforth denoted by $\mathcal{F}(\cdot)$ or $\hat{\cdot}$.
Without loss of generality one may assume that $\hat\psi$ and $\hat\varphi$ have compact support in $\R\times\Z$.
Since $\Phi^\eps$ are centered random variables in the second inhomogenous Wiener chaos, by the fourth moment theorem of Nualart and Peccati \cite{Nualart2005} it suffices to check that the variance and the fourth moment of $\Phi^\eps$ converge to those of $\Phi$.

Concerning the variance, it is clear that $\E(\xi^\eps(\psi))^2\to\|\psi\|_{L^2}^2$ and that $\E\big(\xi^\eps(\psi)\tilde\xi^\eps(\varphi)\big)=0$ by the orthogonality of the first and second homogeneous chaos. To show the convergence of $\E(\tilde\xi^\eps(\varphi))^2$, denote by $\varphi^\eps$ the temporal convolution of $\varphi$ with $\rho^{\eps^2}$ and recall that  $(W^k_t)_{t\in\R}:=\xi\big((s,y)\mapsto\mathbf{1}_{[0,t]}(s)e^{-2\pi ik y}\big)$ are complex standard Brownian motions that are uncorrelated unless $k+k'=0$
(here, for $t<0$, we understand $\mathbf{1}_{[0,t]}:=-\mathbf{1}_{[t,0]}$).
With this notation one has
 \begin{equ}
\partial_x\hat X^\eps(t,k)=ik|k|^\gamma\hat\rho^\eps(k)\int_{-\infty}^te^{-4\pi^2(t-s)k^2}\,dW^k_s,
\end{equ}
and therefore
\begin{equs}
  \tilde\xi^\eps(\varphi)&=-\eps^{\beta_\gamma}\sum_{k\in\Z}\sum_{\ell+m+k=0}\mathbf{1}_{\ell,m\neq 0}\ell m|\ell|^\gamma|m|^\gamma\hat\rho^\eps(\ell)\hat\rho^\eps(m)
  \label{eq:formulaforxi}
\\
&\qquad\times\int_\R \hat\varphi^\eps(t,k)\Big(\int_{-\infty}^te^{-4\pi^2(t-r)\ell^2}\,dW^\ell_r\int_{-\infty}^te^{-4\pi^2(t-u)m^2}\,dW^m_u
-\one_{k=0}\frac{1}{8\pi^2\ell^2}\Big)\,dt.
\end{equs}
Note that the $k=0$ term includes the renormalisation
\begin{equ}
\frac{1}{8\pi^2\ell^2}=\E\Big(\int_{-\infty}^te^{-4\pi^2(t-r)\ell^2}\,dW^\ell_r\int_{-\infty}^te^{-4\pi^2(t-u)m^2}\,dW^m_u\Big),
\end{equ}
where $\ell+m=0$, $\ell,m\neq0$.
By Wick's formula and the isometry of stochastic integrals one has
\begin{equs}
\E (\tilde\xi^\eps(\varphi))^2
&=2\eps^{2\beta_\gamma}\sum_{k\in\Z}\sum_{\ell+m+k=0}|\ell|^{2+2\gamma}|m|^{2+2\gamma}\big|\hat\rho^\eps(\ell)\hat\rho^\eps(m)\big|^2
\\
&\quad\times\int_\R\int_\R \hat\varphi^\eps(t,k)\hat\varphi^\eps(t',-k)
\\
&\qquad\times
\int_{-\infty}^{t\wedge t'} e^{-4\pi^2(t+t'-2r)\ell^2}\,dr\int_{-\infty}^{t\wedge t'} e^{-4\pi^2(t+t'-2u)m^2}\,du\,dt\, dt'
\\
&=\frac{1}{32\pi^4}\eps^{2\beta_\gamma}\sum_{k\in\Z}\sum_{\ell+m+k=0}|\ell|^{2\gamma}|m|^{2\gamma}\big|\hat\rho^\eps(\ell)\hat\rho^\eps(m)\big|^2
\\
&\quad\times\int_\R\int_\R \hat\varphi^\eps(t,k)\hat\varphi^\eps(t',-k)e^{-4\pi^2|t-t'|(\ell^2+m^2)}\,dt\,dt'
\\
&=\frac{1}{32\pi^4}\sum_{k\in\Z}\eps \sum_{\substack{\ell,m\in\eps\Z\\\ell+m+\eps k=0}}|\ell|^{2\gamma}|m|^{2\gamma}\big|\hat\rho(\ell)\hat\rho(m)\big|^2
\\
&\quad\times\int_\R\int_\R \hat\varphi^\eps(t,k)\hat\varphi^\eps(t',-k)\frac{e^{-4\pi^2|t-t'|\frac{\ell^2+m^2}{\eps^2}}}{\eps^2}%% \frac{\ell^2+m^2}{\eps^2}
\,dt\,dt',\label{eq:varrr}
\end{equs}
performing the change of variables $\ell\mapsto\eps\ell$, $m\mapsto \eps m$ in the last step and using the definition of $\beta_\gamma=2\gamma-1/2$.
For any $k$ consider the measures {\color{black}over $\R$}
\begin{equ}
\mu^\eps_k(d\tau)=\frac{1}{32\pi^4}\eps\sum_{\ell+m+\eps k=0}\frac{|\ell|^{2\gamma}|m|^{2\gamma}}{\ell^2+m^2}%\big|\hat\rho(\ell)\hat\rho(m)\big|^2
\big|\hat\rho(\ell)\hat\rho(m)\big|^2
e^{-4\pi^2|\tau|\frac{\ell^2+m^2}{\eps^2}}\frac{\ell^2+m^2}{\eps^2}d\tau.
\end{equ}
The total mass of $\mu^\eps_k$ converges to 
\begin{equ}\label{eq:beautiful-c}
c_{\gamma,\rho}^2:=\frac{1}{128\pi^6}\int_\R|q|^{4\gamma-2}|\hat\rho(q)|^4\,dq,
\end{equ}
which is finite since $\gamma>1/4$ and $\rho$ is smooth. On the other hand, by virtue of $e^{-|x|}\lesssim |x|^{-N}$ for any $N>0$, one easily sees that $\mu^\eps_k([a,b])\to 0$ for any $a>0$ or $b<0$. Thus $\mu^\eps_k\to  c_{\gamma,\rho}^2\delta_0$.
Recalling that for our choice of $\varphi$, the sum over $k$ is finite, the limit in $\eps$ and the sum in $k$ is trivially interchangable, and thus from \eqref{eq:varrr} one gets
\begin{equ}
\lim_{\eps\to 0}\E (\tilde\xi^\eps(\varphi))^2=c_{\gamma,\rho}^2\sum_{k\in\Z}\int_\R |\hat\varphi(t,k)|^2\,dt=c_{\gamma,\rho}^2\|\varphi\|_{L^2}^2
\end{equ}
as desired.

We remark (for later use) that the total mass of $\mu^\eps_k$ is in fact bounded uniformly in $k$, $\eps$, since 
\begin{equs}
\mu^\eps_k(\R)&\lesssim \eps\sum_{\substack{\ell+m+\eps k=0\\|\ell|\geq m}}\frac{|\ell|^{2\gamma}|m|^{2\gamma}}{\ell^2+m^2}\big|\hat\rho(\ell)\hat\rho(m)\big|^2
\lesssim
\eps\sum_{\ell\in\eps\Z}|\ell|^{4\gamma-2}\big|\hat\rho(\ell)\big|^2
\lesssim \int_\R|q|^{4\gamma-2}|\hat\rho(q)|^2\,dq.
\end{equs}
Therefore from \eqref{eq:varrr} one has
\begin{equ}\label{eq:hopefully-true}
\E (\tilde\xi^\eps(\varphi))^2\lesssim \sum_{k\in\Z}\int_\R\varphi^\eps(t,k)(\mu^\eps_k\ast\varphi^\eps)(t,-k)\,dt\lesssim
%\sup_{\eps,k}\mu^\eps_k(\R)
\|\varphi\|_{L^2}^2.
\end{equ}

Next we aim to show that
\begin{equ}
\E(\Phi^\eps)^4\to\E(\Phi)^4=3(\E(\Phi)^2)^2=3\|\psi\|_{L^2}^4+6c_{\gamma,\rho}^2\|\psi\|_{L^2}^2\|\varphi\|_{L^2}^2+3c_{\gamma,\rho}^4\|\varphi\|_{L^2}^4.
\end{equ}
Expanding the power
\begin{equ}
(\Phi^\eps)^4=(\xi^\eps(\psi))^4+4(\xi^\eps(\psi))^3\tilde\xi^\eps(\varphi)
+6(\xi^\eps(\psi))^2(\tilde\xi^\eps(\varphi))^2
+
4\xi^\eps(\psi)(\tilde\xi^\eps(\varphi))^3
+(\tilde \xi^\eps(\varphi))^4,
\end{equ}
all but two terms are trivial: one has $\E(\xi^\eps(\psi))^4\to3\|\psi\|_{L^2}^4$ from Gaussianity, as well as $\E\big(\xi^\eps(\psi)(\tilde\xi^\eps(\varphi))^3\big)=\E\big((\xi^\eps(\psi))^3\tilde\xi^\eps(\varphi)\big)=0$ from the orthogonality of even and odd homogeneous chaoses.

For the remaining two terms
let us introduce some graphical notation for certain random variables and their expectations. The diagrams contain vertices of three types (\begin{tikzpicture}[baseline=-3pt]\draw (0,0) node[noise]{};\end{tikzpicture}$\,$, \begin{tikzpicture}[baseline=-3pt]\draw (0,0)node[intF]{};\end{tikzpicture}$\,$ and \begin{tikzpicture}[baseline=-3pt]\draw (0,0) node[nod]{};\end{tikzpicture}$\,$) and directed edges.
Each vertex of type \begin{tikzpicture}[baseline=-3pt]\draw (0,0) node[noise]{};\end{tikzpicture} and \begin{tikzpicture}[baseline=-3pt]\draw (0,0)node[intF]{};\end{tikzpicture} is assigned a variable $t\in\R$. At vertices of type \begin{tikzpicture}[baseline=-3pt]\draw (0,0) node[nod]{};\end{tikzpicture} we fix the corresponding variable to be $0$.
Furthermore, each edge is directed and is assigned a variable $k\in\Z$, with the constraint that at each vertex of type \begin{tikzpicture}[baseline=-3pt]\draw (0,0)node[intF]{};\end{tikzpicture} the sum of the variables of all incoming edges equals  the sum of the variables of all outcoming edges (``conservation of momentum''). Each edge (directed from vertex $v$ to $w$) corresponds to a function $g:\R\times\Z\to\mathbb{C}$ that is evaluated at the $k$ coordinate of the edge and at the difference of the $t$ coordinates of vertex $w$ and that of vertex $v$.
There are two types of edges: 
\begin{tikzpicture}[baseline=-3pt]
\draw[HK] (0,0)--(0.5,0);
\end{tikzpicture}
corresponds to $g(t,k):=e^{-4\pi^2 tk^2}\one_{t\geq0} ik|k|^\gamma\hat\rho^\eps(k)$,
while
\begin{tikzpicture}[baseline=-3pt]
\draw[test] (0,0)--(0.5,0);
\end{tikzpicture}
corresponds to either $g(t,k)=\hat\varphi^\eps(t,-k)$ or $g(t,k)=\hat \psi^\eps(t,-k)$.
Crucially, the second type of edge is nonzero only for finitely many $k$-s.
Vertices \begin{tikzpicture}[baseline=-3pt]\draw (0,0) node[noise]{};\end{tikzpicture} always have degree $1$, and to them is attached the  stochastic integral $dW^k_t$, where $k$ is that of the unique edge connected to that vertex.
The value of a diagram is obtained by integrating/summing all the variables.

For example,
\begin{equ}
\xi^\eps(\psi)=
\begin{tikzpicture}[baseline=8pt]
\draw[test](0,0.8)-- (0,0)node[nod]{};
\draw (0,0.8) node[noise]{};
\end{tikzpicture},
\qquad
\tilde\xi^\eps(\varphi)=
\eps^{\beta_\gamma}\left(\,
\begin{tikzpicture}[baseline=10pt]
\draw[test] (0,0.65)-- (0,0)node[nod]{};
\draw[HK] (-0.35,1)--(0,0.65) ;
\draw (-0.35,1) node[noise]{};
\draw[HK] (0.35,1)--(0,0.65)node[intF]{};
\draw (0.35,1)node[noise]{};
\end{tikzpicture}
-\,\,
\begin{tikzpicture}[baseline=10pt]
\draw[test] (0,0.65)--(0,0)node[nod]{};
\draw[HK] (0,1.1) to [out=210, in=150, looseness=2] (0,0.65);
\draw[HK] (0,1.1) to [out=-30, in=30, looseness=2] (0,0.65);
%\draw (0,0.65)node[int]{} to [out=30, in=-30, looseness=2] (0,1.1)node[intF]{};
\draw (0,0.65)node[intF]{};
\draw (0,1.1)node[intF]{};
\end{tikzpicture}\,\right),
\qquad
\E(\tilde\xi^\eps(\varphi))^2=2\eps^{2\beta_\gamma}
\begin{tikzpicture}[baseline=10pt]
\draw[test] (0,0.65)--(0,0)node[nod]{};
\draw[test] (1,0.65)--(1,0)node[nod]{};
\draw[HK] (1,1.1) to [in=45, out=180, looseness=1] (0,0.65);
\draw[HK] (1,1.1) to [out=0, in=45, looseness=2] (1,0.65);
\draw[HK]  (0,1.1) to [in=135, out=0, looseness=1] (1,0.65);
\draw[HK] (0,1.1) to [out=180, in=135, looseness=2] (0,0.65);
\draw (1,1.1)node[intF]{};
\draw (0,1.1)node[intF]{};
\draw (0,0.65) node[intF]{};
\draw (1,0.65) node[intF]{} ;
\end{tikzpicture}\,.
\end{equ}
The reader is invited to compare the graphical representation of $\tilde\xi^\eps({\color{black}\varphi})$ (middle drawing) with its formula \eqref{eq:formulaforxi}.
{\color{black}Similarly, one can compare the last drawing with the first equality in \eqref{eq:varrr}.}
Then by Wick's formula
\begin{equs}
\E&\big((\xi^\eps(\psi))^2(\tilde\xi^\eps(\varphi))^2\big)
=
2\eps^{2\beta_\gamma}
\begin{tikzpicture}[baseline=10pt]
\draw[test]  (-1,1) to [in=105,out=210,looseness=1] (-1.2,0)node[nod]{};
\draw[test] (-1,1) node[intF]{} to [out=-30,in=75,looseness=1](-0.8,0)node[nod]{};
\draw[test] (0,0.65)--(0,0)node[nod]{};
\draw[test] (1,0.65)--(1,0)node[nod]{};
\draw[HK] (1,1.1) to [in=45, out=180, looseness=1] (0,0.65);
\draw[HK] (1,1.1) to [out=0, in=45, looseness=2] (1,0.65);
\draw[HK]  (0,1.1) to [in=135, out=0, looseness=1] (1,0.65);
\draw[HK] (0,1.1) to [out=180, in=135, looseness=2] (0,0.65);
\draw (1,1.1)node[intF]{};
\draw (0,1.1)node[intF]{};
\draw (0,0.65) node[intF]{};
\draw (1,0.65) node[intF]{} ;
\end{tikzpicture}
+
8\eps^{2\beta_\gamma}
\begin{tikzpicture}[baseline=10pt]
\draw[test] (0,1.1)node[intF]{} to[in=90,out=180,looseness=1.2] (-1.2,0)node[nod]{} ;
\draw[test] (-0.4,0.9) to[in=90,out=180,looseness=1] (-0.8,0)node[nod]{};
\draw[HK]  (-0.4,0.9) node[intF]{}  to[out=0,in=135,looseness=1] (0,0.65)node[int]{};
\draw[test] (0,0.65)--(0,0)node[nod]{};
\draw[test] (1,0.65)--(1,0)node[nod]{};
\draw[HK] (1,1.1) to [in=45, out=180, looseness=1] (0,0.65);
\draw[HK] (1,1.1) to [out=0, in=45, looseness=2] (1,0.65);
\draw[HK]  (0,1.1) to [in=135, out=0, looseness=1] (1,0.65);
\draw (1,1.1)node[intF]{};
\draw (0,1.1)node[intF]{};
\draw (0,0.65) node[intF]{};
\draw (1,0.65) node[intF]{} ;
\end{tikzpicture}\,.
\end{equs}
The first term is simply $\E\big((\xi^\eps(\psi))^2\big)\E\big((\tilde\xi^\eps(\varphi))^2\big)$, which converges to $c_{\gamma,\rho}^2\|\psi\|_{L^2}^2\|\varphi\|_{L^2}^2$.
 As for the second diagram, notice that the sum over the $k$ variable collapses to a finite sum.
This is due to the fact that $\psi,\varphi$ have finitely many non-zero Fourier modes, and to the conservation rule of momentum at vertices of the diagram. As for the integration on the time variables in the diagram, recall that the function $g$ associated to the lines decays exponentially in the time difference between its endpoints, uniformly in $\eps$. From this fact, one easily sees that the value of this diagram is $\eps^{2\beta_\gamma}$ times a quantity uniformly bounded in $\eps$ and since $\beta_\gamma>0$ this diagram tends to zero.  

Similarly, one has
\begin{equs}
\E(\tilde \xi^\eps(\varphi))^4
&=12\eps^{4\beta_\gamma}
\begin{tikzpicture}[baseline=10pt]
\draw[test] (0,0.65)--(0,0)node[nod]{};
\draw[test] (1,0.65)--(1,0)node[nod]{};
\draw[HK] (1,1.1) to [in=45, out=180, looseness=1] (0,0.65);
\draw[HK] (1,1.1) to [out=0, in=45, looseness=2] (1,0.65);
\draw[HK]  (0,1.1) to [in=135, out=0, looseness=1] (1,0.65);
\draw[HK] (0,1.1) to [out=180, in=135, looseness=2] (0,0.65);
\draw (1,1.1)node[intF]{};
\draw (0,1.1)node[intF]{};
\draw (0,0.65) node[intF]{};
\draw (1,0.65) node[intF]{} ;
\draw[test] (2,0.65)--(2,0)node[nod]{};
\draw[test] (3,0.65)--(3,0)node[nod]{};
\draw[HK] (3,1.1) to [in=45, out=180, looseness=1] (2,0.65);
\draw[HK] (3,1.1) to [out=0, in=45, looseness=2] (3,0.65);
\draw[HK]  (2,1.1) to [in=135, out=0, looseness=1] (3,0.65);
\draw[HK] (2,1.1) to [out=180, in=135, looseness=2] (2,0.65);
\draw (3,1.1)node[intF]{};
\draw (2,1.1)node[intF]{};
\draw (2,0.65) node[intF]{};
\draw (3,0.65) node[intF]{} ;
\end{tikzpicture}
+48\eps^{4\beta_\gamma}
\begin{tikzpicture}[baseline=10pt]
\draw[test] (0,0.65)--(0,0)node[nod]{};
\draw (0.3,0.4) node{\tiny $k_1$};
\draw (1.3,0.4) node{\tiny $k_2$};
\draw (2.3,0.4) node{\tiny $k_3$};
\draw (3.3,0.4) node{\tiny $k_4$};
\draw[test] (1,0.65)--(1,0)node[nod]{};
\draw[test] (2,0.65)--(2,0)node[nod]{};
\draw[test] (3,0.65)--(3,0)node[nod]{};
\draw[HK] (0.5,0.9) to[in=45,out=180,looseness=1](0,0.65) node[intF]{}; 
\draw[HK](0.5,0.9)node[intF]{} to[out=0,in=135] (1,0.65) node[intF]{};
\draw[HK] (1.5,0.9) to[in=45,out=180,looseness=1] (1,0.65)node[intF]{}; 
\draw[HK](1.5,0.9)node[intF]{} to[out=0,in=135] (2,0.65)node[intF]{};
\draw[HK] (2.5,0.9) to[in=45,out=180,looseness=1] (2,0.65)node[intF]{}; 
\draw[HK](2.5,0.9)node[intF]{} to[out=0,in=135] (3,0.65)node[intF]{};
\draw[HK] (1.5,1.2) to[in=135, out=180, looseness=1] (0,0.65)node[intF]{}; 
\draw[HK] (1.5,1.2)node[intF]{} to[out=0,in=45,looseness=1] (3,0.65)node[intF]{};
\end{tikzpicture}\,.
\end{equs}
The first term is simply $3(\E(\tilde\xi^\eps(\varphi))^2)^2$, which converges to $3c_{\gamma,\rho}^4\|\varphi\|_{L^2}^4$.
In the second diagram we indicated the $k$ variables of the
\begin{tikzpicture}[baseline=-3pt]
\draw[test] (0,0)--(0.5,0);
\end{tikzpicture}
edges,
observe that once again the sum over $k_1,\dots,k_4$ contains only finitely many terms (because $\varphi$ has finite Fourier support).
It is clear that the momentum of one remaining edge, say $\ell$, uniquely determines all the others because of conservation of momentum.
Recalling the form of the functions associated to the edges, it is easy to see that the second term above is upper bounded by a constant
times 
  \begin{equ}
    \eps^{4\beta_\gamma}\sum_\ell |\hat \rho^\eps(\ell)|^8 |\ell|^{8(\gamma+1)}\times \frac1{\ell^8}\times \frac1{\ell^6}.
  \end{equ}
  In fact, the term $|\ell|^{8(\gamma+1)}$ comes from prefactors in the function $g$ associated to the eight lines of type 
\begin{tikzpicture}[baseline=-3pt]
\draw[HK] (0,0)--(0.5,0);
\end{tikzpicture},
the factor $\ell^{-8}$ comes from the integrals with respect of the time variable of  the four black vertices not connected to a \begin{tikzpicture}[baseline=-3pt]
\draw[test] (0,0)--(0.5,0);
\end{tikzpicture} line and the factor
  $\ell^{-6}$ from the integral over the three time differences between the four vertices connected to a \begin{tikzpicture}[baseline=-3pt]
\draw[test] (0,0)--(0.5,0);
\end{tikzpicture} line.
  Since the cut-off function $\hat\rho^\eps(\ell)$ effectively restricts the sum to values $|\ell|\lesssim \eps^{-1}$, it is easily seen that the value of the term is of order $\eps^3$ and in particular tends to zero.

It remains to show tightness, which one may do component-wise.
Let $(\chi_j)_{j\geq-1}$ form a smooth dyadic partition of unity: that is, $\chi_{-1}$ is smooth and supported on the ball of radius $1$, and $\chi_0$ is smooth and supported on the annulus between radii $1/2$ and $2$, for all $a\in\R$ and $j\geq 1$ one has $\chi_j(a)=\chi_0(2^{-j}a)$ and for all $a\in \R$, $\sum_{j\geq -1}{\color{black}\chi_j( a)}=1$.
We denote $\Delta_{j}f=\mathcal{F}^{-1}(\chi_j\hat f)$.
One then has the following tighness criterion for random fields. It follows from e.g. {\color{black}\cite[Lem.~9-10]{MW}}, Nelson's estimate, and Besov embeddings.
\begin{lemma}\label{lem:kolmogorov}
Let $k\geq 0$ and $(V^\eps)_{\eps\in(0,1]}$ be a family of continuous random processes indexed by $t\in[0,T]$ and $x\in\T$, such that they belong to the $k$-th nonhomogeneous Wiener chaos. Suppose that there exist constants $a>0$, $b\in \R$, and $C<\infty$ such that for all $0\leq s<t\leq T$ and all $j\geq -1$ one has
\begin{equ}\label{eq:kolmogorov}
\sup_{\eps\in(0,1]}\sup_{x\in\T}\E \Big(\Delta_j\big(V^\eps(t,\cdot)-V^\eps(s,\cdot)\big)(x)\Big)^2\leq C |t-s|^{2a} 2^{-2jb}.
\end{equ}
Then for all $\kappa>0$, $(V^\eps)_{\eps\in(0,1]}$ is tight in $\cC^{a-\kappa}([0,T]; \cC^{b-\kappa}(\T))$.
\end{lemma}
The tightness of $\bar X^\eps$ is very well-known, let us give the proof anyway.
To show \eqref{eq:kolmogorov} with $V^\eps=\bar X^\eps$, fix $s,t,x,j$, and consider the function
\begin{equ}
\varphi(r,y)=\Delta_j\big(\cP_0(t-r,\cdot)-\cP_0(s-r,\cdot)\big)(x-{\color{black}y}).
\end{equ}
Then
\begin{equ}
\E \Big(\Delta_j\big(\bar X^\eps(t,\cdot)-\bar X^\eps(s,\cdot)\big)(x)\Big)^2=\E(\xi^\eps(\varphi))^2=\|\varrho^\eps\ast\varphi\|_{L^2}^2\leq\|\varphi\|_{L^2}^2.
\end{equ}
One writes (using the elementary inequalites $e^{-4\pi^2x}\leq e^{-x}$, $e^{-4\pi^2x}-e^{-4\pi^2y}\leq e^{-x}-e^{-y}$ for $0\leq x\leq y$, to save some space)
\begin{equs}
\|\varphi\|_{L^2}^2&=\int_s^t\|\Delta_j \cP_0(t-r,\cdot)\|_{L^2(\T)}^2\,dr
    +\int_{-\infty}^s\|\Delta_j \big(\cP_0(t-r,\cdot)-\cP_0(s-r,\cdot)\big)(\cdot)\|_{L^2(\T)}^2\,dr
    \\&\leq\int_{s}^{t}\sum_{k\in\Z}(\chi_{j}(k))^2 e^{-(t-r)k^{2}}dr + \int_{-\infty}^{s}\sum_{k\in\Z}(\chi_{j}(k))^2 e^{-(s-r)k^2}(1-e^{(t-s)k^2})^2dr.
%    \\
%    &\leq \int_{s}^{t}\sum_{k\in\Z}(\chi_{j}(k))^2 (t-r)^{-1+\kappa}k^{-2+2\kappa}dr + \int_{-\infty}^{s}\sum_{k\in\Z}(\chi_{j}(k))^2 e^{-(s-r)k^2}(1-e^{(t-s)k^2})^2dr
\end{equs}
Using that for any $\theta\geq 0$, $e^{-x}\lesssim x^{-\theta}$ uniformly over $x\geq 0$ and that for any $\theta\in[0,1]$, $1-e^{-x}\lesssim x^\theta$ uniformly over $x\geq 0$, as well as the defining properties of $\chi_j$ one gets that for any $\kappa>0$
\begin{equ}\label{eq:HK-bound}
\|\varphi\|_{L^2}^2\lesssim |t-s|^{\kappa}|2^{-j(1-3\kappa)}.
\end{equ}
Applying Lemma \ref{lem:kolmogorov}, with $\kappa$ small enough (both therein and in \eqref{eq:HK-bound}) gives the required bound.

Very similarly, the bound \eqref{eq:HK-bound} holds for the function
\begin{equ}
\varphi(r,y)=\one_{r\geq 0}\Delta_j\big(\cP(t-r,\cdot)-\cP(s-r,\cdot)\big)(x-{\color{black}y}).
\end{equ}
This and \eqref{eq:hopefully-true}  implies \eqref{eq:kolmogorov}
with $V^\eps=Y^\eps$, which yields the tightness of $Y^\eps$.\qed

% $a=\kappa$, $b=\alpha+\kappa$, with sufficiently small $\kappa>0$, we use \eqref{eq:hopefully-true} with 
%It is well known that 
%\begin{equ}
%\|\varphi\|_{L^2}^2\lesssim |t-s|^{\kappa}|2^{-j+2\kappa j},
%\end{equ}
%for any $\kappa\in[0,1]$, which concludes the proof.\mate{I can write this. Quite standard}

\subsection{A short recap of regularity structures}
We first provide a very brief recap of some of the main concepts in the theory of regularity structures introduced in \cite{H0}.
By no means do we (or can we) aim for completeness, but hopefully this recap helps the reader to follow the high level strategy.

Consider a family of noises (in our case, $\eta^\eps:=\eps^{(2/3)\beta_\gamma-\kappa}D^\gamma\xi^\eps$).
Consider a family of equations (in our case \eqref{eq:bigsystem} below) driven by one of these noises, involving a set of parameters (in our case $\hat C$ below). One writes a single analogous equation (in our case \eqref{eq:bigsystem-modelled} below) for functions that take values in a so-called regularity structure $\cT$. The parameters do not appear in the abstract equation.

The set $\cT$ is a vector space whose basis elements can be represented as the linear span of a finite family of rooted trees, where leafs (in symbol: $\Xi$) represent the noises, and edges (in symbol: $\cI$ or $\cI'$) represent convolution with a kernel $K$ or its spatial derivative. Here $K$ is essentially the heat kernel, modulo to a convenient truncation. A model $\PPi$ is a linear map from $\cT$ to the space of distributions (for us, actually always to smooth functions) that makes this representation precise in that $\PPi(\Xi)=\eta^\eps$ for some $\eps\in[0,1]$ and $\PPi(\cI\tau)=K\ast\PPi(\tau)$ for any $\tau\in \cT$ such that $\cI\tau\in\cT$. For each model there is map $\cR$ that maps (sufficiently nice) $\cT$-valued functions to real-valued ones. The abstract equation is built from this map, and crucially, both $\cR$ and the (unique) solution of the equation is a continuous function of $\PPi$, equipped with an appropriate metric.

If $\eps>0$ and the model further satisfies $\PPi(\tau\bar\tau)=(\PPi\tau)(\PPi\bar\tau)$ (in terms of trees, the multiplication means joining them at the root), this model is called the canonical model. The canonical model has the property that the $\cR$ of the solution of the abstract equation is the solution of the equation one started with, driven by $\eta^\eps$ and $\hat C=0$. 

Unfortunately, the canonical models typically do not converge as $\eps\to 0$. The so-called BPHZ models are built from algebraic deformations of the canonical models, that have two very nice properties: they do converge as $\eps\to0$, and the $\cR$ of the solution of the abstract equations still satisfies an equation of the kind that one started with, but with a different choice of parameters $\hat C$. Typically the deformation of the parameters diverge as $\eps\to 0$, this is the source of renormalisation.
By the convergence of the BPHZ models as $\eps\to 0$, by the continuity of the solution map and $\cR$, one obtains the convergence of the solutions of the renormalised equations as $\eps\to 0$.

\subsection{Proof of (b)}\label{sec:proof(b)}
The goal of the present section is to verify that the equation for the remainder $Z^\eps$ falls in the scope of the black box {\color{black}solution theory of regularity structures developed in} \cite{H0, CH, BHZ, BCCH}.
In the spirit of \cite{GHM} {\color{black}(see Eqs. (1.12)-(1.15) therein and the surrounding discussion), if the multiplication with powers of $\eps$ falls on a noise, one can view this as if the driving noise were more regular (without introducing any further operators like $\mathcal{E}$ in \cite{HQ})}.
{\color{black}This is made precise in Proposition \ref{prop:CM} below.}
 With this improved homogeneity assignment the \emph{super-regularity} assumption of \cite{CH}, that failed for the original equation \eqref{eq:KPZ}, is satisfied.

In principle the larger $\gamma$ is, the easier the problem becomes, but since larger $\gamma$ changes some power counting arguments slightly, we replace it in the regularity structure by $\gamma_0=\gamma\wedge1$. Take $\kappa>0$, which will be chosen to be sufficiently small.
Set 
\begin{equs}
%\mu_1&=-3/2-{\gamma_0}-2\kappa+\beta_{\gamma_0},
%\\
\mu&=-3/2-{\gamma_0}-3\kappa+(2/3)\beta_{\gamma_0}.
\end{equs}
Take a smooth Gaussian noise $\eta$ with
%pair of noises $\Xi=(\eta_1,\eta_2)$ with
regularity $\mu$ (in a sense made precise below). Assume that 
\begin{equ}
%\hat X_i(t,x)=\int_{(-\infty,0]\times\T} \cP(-s,\cdot-y)\eta_i(s,y)\,ds\,dy\in C^{\mu_i+2}(\T).....
\hat\zeta(\cdot)=\int_{(-\infty,0]\times\T} \cP(-s,\cdot-y)\eta(s,y)\,ds\,dy\in C^{\mu+2}(\T).
\end{equ}
%Denote $\hat\zeta_i...$
%Suppose that $\eta^i_\eps$ converge in $\cC^{\mu_i}$ {\color{red}in a sufficiently strong sense}. Define
%\begin{equs}
%\cX_i^\eps(t,x)&=\int_{[0,t]\times\T}\cP(t-s,x-y)\eta_i^\eps(s,y)\,ds\,dy,
%\\
%\cY_i^\eps(t,x)&=\int_{[0,t]\times\T}\cP(t-s,x-y)\big((\d_x \cX_i^\eps(s,y))^2-\E (\d_x \cX_i^\eps(0,0))^2 \big)\,ds\,dy.
%\end{equs}
Take a parameter $c\in\R^5$. 
The noise $\eta$ (more precisely, its law) and $c$ determines a certain \emph{renormalisation constant} {\color{black}$\hat C\in\R^2$}, also detailed below.
%The latter should not be thought as a free parameter, it will be determined by $\Xi$ (more precisely, its law) and $c$, see also below.
% Together with $\eta_1,\eta_2$ (more precisely, their law), they determine constants $\cC_i,\bar \cC$, see \eqref{eq:BPHZ-constants} below, but whose value is not particularly important at this stage.
Consider the system of equations 
%(the first of which is simply satisfied by definition)
\begin{equs}[eq:bigsystem]
(\dd_t -\Delta)\hat X &= \eta,\qquad &\hat X(0,x)&=\hat\zeta,
\\
(\dd_t -\Delta)\hat Y &= (\dd_x \hat X)^2-{\color{black}\hat C_1},\qquad &\hat Y(0,x)&=0,
\\
(\dd_t-\Delta)\hat Z&=c_1(\dd_x \hat Z)^2+2c_2(\dd_x \hat Z)(\dd_x \hat X)+2c_3(\dd_x \hat Z)(\dd_x \hat Y)
\\
&\qquad+2c_4(\dd_x \hat X)(\dd_x \hat Y)+c_5(\dd_x \hat Y)^2
-{\color{black}\hat C_2},\qquad &\hat Z(0,x)&=\psi(x).
\end{equs}
%The constants $\cC_i$ and $\bar\cC$ will be determined by the choice of $\eta_j$ and $c_k$, and therefore should not be viewed as separate parameters. {\color{red}(could be phrased better..)}
When emphasizing the dependence on $\eta$ and $c$, we write $\hat X(\eta)$, $\hat Y(\eta)$, $\hat Z(\eta,c)$, ${\color{black}\hat C}(\eta,c)$.
The sketch of the proof is as follows. Below we invoke the theory of regularity structures to verify that if the renormalisation constant ${\color{black}\hat C}$ is defined appropriately, then $\hat Z$ is continuous in both coordinates (so are $\hat X$ {\color{black}and} $\hat Y$, but that is easy to see without any deep theory).
Set, {\color{black}for any $\eps\in[0,1]$}
\begin{equs}[eq:choice]
%\Xi^\eps:&=(\eps^{\beta_\gamma-\kappa}D^\gamma\xi^\eps,\eps^{(2/3)\beta_\gamma-\kappa}D^\gamma\xi^\eps),\\
\eta^\eps:=\eps^{(2/3)\beta_\gamma-\kappa}D^\gamma\xi^\eps,\quad c^\eps:&=(\eps^{\beta_\gamma},\eps^{(1/3)\beta_\gamma+\kappa},\eps^{(2/3)\beta_\gamma+2\kappa},\eps^{3\kappa},\eps^{(1/3)\beta_\gamma+4\kappa}).
\end{equs}
Clearly, % $\hat X_1(\Xi^\eps)=\eps^{\beta_\gamma-\kappa}X^\eps$,
$\hat X(\eta^\eps)=\eps^{(2/3)\beta_\gamma-\kappa}X^\eps$.
Furthermore, if  
%it is verified that the values
%$\cC_1(\Xi^\eps)\eps^{-2\beta_\gamma+2\kappa}$ and $\cC_2(\Xi^\eps)\eps^{-(4/3)\beta_\gamma+2\kappa}$ coincide, then
we take ${\color{black}C_1^\eps}=\eps^{-(1/3)\beta_\gamma+2\kappa}{\color{black}\hat C_1}(\eta^\eps)$ in \eqref{eq:system0}, one also has
%\begin{equ}\label{eq:some-reno-correct}
%\cC_1(\Xi^\eps)=\eps^{2\kappa}C_0^\eps,\qquad \cC_2(\Xi^\eps)=
%\end{equ} 
%also $\hat Y_1(\Xi^\eps)=\eps^{\beta_\gamma-2\kappa}Y^\eps$,
$\hat Y(\eta^\eps)=\eps^{(1/3)\beta_\gamma-2\kappa}Y^\eps$.
It will of course have to be verified that this choice of ${\color{black}C_1^\eps}$ satisfies the condition of Theorem \ref{thm:main2} (a).
Finally, if in \eqref{eq:Z} one takes ${\color{black}C^\eps_2=\hat C_2(\eta^\eps,c^\eps)}$, then $Z^\eps=\hat Z(\eta^\eps,c^\eps)$. Using the aforementioned continuity, $\hat Z(\eta^\eps,c^\eps)\to\hat Z(0,0)=\bar Z$ as claimed.

%\begin{remark}
%It would be somewhat natural to also include a ``noise'' with regularity $-3/2-\gamma-2\kappa+(3/4)\beta_\gamma$ to treat the last term in \eqref{eq:Z}, given that it includes $4$ instances of $D^\gamma\xi^\eps$ and $3$ instances of $\eps^{\beta_\gamma}$. But since the power counting works out as it is, we refrain from making the system even larger.
%\end{remark}

%We then have (provided the renormalisation constants are chosen appropriately)
%\begin{equ}\label{eq:relation1}
%Z^\eps=\cZ\Big(\eps^{\beta_\gamma-\kappa}D^\gamma\xi^\eps,\eps^{(2/3)\beta_\gamma-\kappa}D^\gamma\xi^\eps,\eps^{\beta_\gamma},\eps^{\kappa},\eps^{2\kappa},\eps^{3\kappa},\eps^{(1/3)\beta_\gamma+4\kappa}\Big).
%\end{equ}
%Before going into the details, let us highlight the gain in the above formulation: 
%is that it fits into the black box theory of \cite{H0,CH,BHZ}.
%it is, in essence, a rewriting of the KPZ equation with two noises, where each noise is above the threshold of \cite{Hoshino}:
%\begin{equ}
%\mu_2>\mu_1=-2+\gamma-2\kappa>-7/4,
%\end{equ}
%{\color{red}Not quite, $\mu_2<\mu_1...$ Some thing has to be changed here}
%provided $\kappa$ is small enough.
%The main purpose of this formulation is that the system \eqref{eq:bigsystem} \emph{does} satisfy the super-regularity condition \cite[Def.~2.28]{CH}, and therefore  the general theory can be applied. We now make this precise. For simplicity we assume $\gamma<1$ first, the treatment of the case $\gamma\geq 1$ is addressed in Remark \ref{rem:gamma>1} below.

The system \eqref{eq:bigsystem} can be formulated via the regularity structure in \cite{Hoshino}: Indeed, the nonlinearities are inherited from the KPZ equation and the regularity of the noise is $\mu>-7/4$ if $\kappa$ is small enough.
In particular, in this regularity structure
the super-regularity condition \cite[Def.~2.28]{CH} (see also the second condition in \cite[Asn.~2.31]{HS})
is satisfied.

The construction of models for regularity structures is established in far-reaching generality in \cite{CH}.
We also invoke \cite{HS}, which gives a new proof of the main results of \cite{CH}, and which, more importantly for us, does not assume the singularity of the covariance of the noise to have its singularity located only at the space-time origin. This is useful in our setting, where the singularity is located on the whole line $\{t=0\}$. Following the notation in \cite[Def.~2.24]{HS}, denote $\mathrm{scal}:=-\gamma_0-\kappa+(2/3)\beta_{\gamma_0}$. Denote also by $H_{\parr}^\alpha(\R\times\T)$ the Bessel potential spaces with parabolic scaling equipped with the obvious norm.

\begin{proposition}\label{prop:CM}
There exist constants $(c_\eps)_{\eps\in(0,1]}$ with $c_\eps\to0$ such that for all $\eps\in(0,1]$ the Cameron-Martin norm $|\cdot|_\eps$ of $\eta^\eps$ satisfies
\begin{equ}
|\cdot|_\eps\leq c_\eps\|\cdot\|_{H_{\parr}^{\mathrm{scal}}(\R\times\T)}.
\end{equ}
\end{proposition}
\begin{proof}
Decomposing $\xi$ into $1$-dimensional white noises $\xi=\xi_t\otimes\xi_x$, one has the decomposition $\eta^\eps=(\rho^{\eps^2}\ast\xi_t)\otimes\big(\eps^{(2/3)\beta_\gamma-\kappa}D^\gamma\rho^\eps\ast\xi_x\big)$, and therefore the Cameron-Martin space also tensorises.
Since $\mathrm{scal}+\kappa<0$, one has $L^2(\R)\otimes H^{\mathrm{scal}+\kappa}(\T)\subset H^{\mathrm{scal}+\kappa}_{\parr}(\R\times\T)$.
Trivially, the Cameron-Martin norm of $(\rho^{\eps^2}\ast\xi_t)$ is controlled by the $L^2(\R)$ norm.
So it suffices to show that the map $f\mapsto \eps^{(2/3)\beta_\gamma-\kappa}D^\gamma\rho^\eps\ast f$ is uniformly bounded from $L^2(\T)$ to $H^{\mathrm{scal}+\kappa}(\T)$. This simply follows from $\gamma_0\leq \gamma$ and the inequality $\eps^{(2/3)\beta_\gamma-\kappa}\hat\rho^\eps(k)\lesssim |k|^{-(2/3)\beta_\gamma+\kappa}$.
\end{proof}
By the proposition, the noises $\eta^\eps$ satisfy the spectral gap inequality condition \cite[Def.~2.20]{HS} (see also \cite[Rem.~2.22]{HS})  with exponent $\mathrm{scal}$, uniformly in $\eps$ (in fact, with vanishing constant). Since $0>\mathrm{scal}>\mu+3/2$, the first condition in \cite[Asn.~2.31]{HS} is also satisfied.

The main results of \cite{CH,HS}
then imply that (with a fixed spatially symmetric truncation $K$ of the heat kernel as in \cite[Sec.~5]{H0}) the BPHZ lifts $\PPi^\eps$ of $\eta^\eps$ converge in probability
 in the natural topology of models (see \cite[Eq.~(2.17)]{H0}) to a model $\PPi^0$. In fact, $\PPi^0$ is the trivial model that is identically $0$ on the non-polynomial part of the regularity structure but we do not use anything from the limit other than that it exists.

For any of admissible model one can formulate the abstract counterpart of \eqref{eq:bigsystem} :
\begin{equs}[eq:bigsystem-modelled]
\tilde X&=\tilde \cP(\Xi)+{\color{black}E},
%\\
%\tilde X_2 &=\cI_{\tx_2}\<Eta2>+
\\
\tilde Y &= \tilde\cP\big((\dd_x \tilde X)^2\big),
\\
\tilde Z&=\tilde\cP\Big(c_1(\dd_x \tilde Z)^2+2c_2(\dd_x \tilde Z)(\dd_x \tilde X)+2c_3(\dd_x \tilde Z)(\dd_x \tilde Y)
\\&\qquad\qquad+2c_4(\dd_x \tilde X)(\dd_x \tilde Y)+c_5(\dd_x \tilde Y)^2\Big)+\tilde\Psi,
\end{equs}
where $\tilde \cP$ is the abstract heat semigroup (in which for convenience we understood ``everything'' to be included: the multiplication with indicator of positive times, the abstract convolution with $K$, and the convolution with the smooth remainder, see \cite[Sec.~7.1]{H0}),
$E$ and $\tilde\Psi$ are the Taylor lifts of the functions $t,x\mapsto(\cP(t,\cdot)\ast\varphi)(x)$ with the choices $\varphi=\hat\zeta$ and $\varphi=\psi$, respectively. Deriviation and multiplication are all understood in the abstract sense.
For any model \eqref{eq:bigsystem-modelled} has a unique solution
\begin{equ}
(\tilde X,\tilde Y,\tilde Z)\in \cD^{\infty,\mu+2-\kappa}\times\cD^{\infty,2\mu+4-3\kappa}\times \cD^{{\color{black}2-\kappa},\theta}.
\end{equ}
% whose first two coordinates $\tilde X$ and $\tilde Y$ are given trivially and belong to the modelled distribution spaces $\cD^{\infty,\mu+2-\kappa}$ and $\cD^{\infty,2\mu+4-\kappa}$, respectively.
%The last coordinate $\tilde Z$ belongs to the modelled distribution space $\cD^{2-\kappa,\theta}$. 
%%%%%%%%%%%%%%%%%%%%
%		some tedious and not so interesting calculations are done here to verify this. just in case it's needed...
%%%%%%%%%%%%%%%%%%%%
%As a sanity check, let us verify the map given by right-hand side of the last equation of \eqref{eq:bigsystem-modelled}
%The solutions take place in the modelled distribution space 
%is well-defined on this space and gains in the exponents of both of regularity and initial singularity. Indeed, the the worst term inside the $\tilde\cP$ in terms of regularity is $(\d_x \tilde Z)(\d_x \tilde X_1)\in\cD^{\gamma_0-4\kappa,-2+\gamma_0+\theta}$, while the worst term in terms of initial singularity is $(\d_x \tilde Z)^2\in\cD^{,2\theta-2}$
%If $\kappa>0$ is sufficiently small, then all regularities are therefore positive (therefore the reconstruction is unique and the convolution has more than $2-\kappa$ regularty) and all initial singularities are greater than $-2+2\theta$ (and therefore the corresponding exponent of the convolution is higher than $\theta$).

By the main results of \cite{H0} these solutions are continuous both with respect to the model and with respect to $c$. So by making the choice \eqref{eq:choice} and then taking the corresponding BPHZ model, the resulting solutions $(\tilde X^\eps,\tilde Y^\eps,\tilde Z^\eps)$ converge as $\eps\to 0$. By the continuity of reconstruction operator $\cR$, the same holds for $\cR(\tilde X^\eps,\tilde Y^\eps,\tilde Z^\eps)$.
Since $c^0= 0$ {\color{black}(see \eqref{eq:choice})}, one has $\tilde Z^0=\tilde\Psi$ and thus $\cR\tilde Z^0=\bar Z$.
%Since the limiting model is identically $0$, its BPHZ renormalisation is also trivial. 

It remains to verify that $\cR(\tilde X^\eps,\tilde Y^\eps,\tilde Z^\eps)$ solves \eqref{eq:bigsystem} and therefore coincides with $(\hat X^\eps,\hat Y^\eps,\hat Z^\eps)$. 
In the equation for $\cR\tilde X^\eps$ there is no renormalisation so there is nothing to discuss. The equation for $\cR\tilde Y^\eps$ is easily seen to be renormalised by the constant ${\color{black}\hat C_1(\eta^\eps)}=\E(\partial_x K\ast\eta^\eps)^2$.
The constant ${\color{black}C_1^\eps}:=\eps^{-(1/3)\beta_\gamma+2\kappa}{\color{black}\hat C_1(\eta^\eps)}$ then satisfies the condition in (a): indeed, this is equivalent to saying that
\begin{equ}
\eps^{\beta_\gamma}\big(\E(\partial_x P_0\ast D^\gamma\xi^\eps)^2-\E(\partial_x K\ast D^\gamma\xi^\eps)^2\big)\to 0,
\end{equ}
which one can verify by observing that the difference of the two expectations above is bounded by $\eps^{(1-\gamma)\wedge 0-\kappa}$.
The fact that the equation for $\cR\tilde Z^\eps$ is also renormalised by a constant (as opposed to further functionals of the solution) is essentially verified in \cite{Hoshino}. One can also see this by noting that all trees in the regularity structure on which the BPHZ renormalisation character does not vanish belong to $T_0$, the set of binary trees without polynomial decorations. Indeed, any tree with negative homogeneity that does not belong to $T_0$ has either exactly one vertex with one outgoing edge or exactly one vertex with a polynomial decoration, both of which yield vanishing renormalisation due to the spatial antisymmetry of the functions $\partial_x K(t,x)$ and $x$ (see \cite[Prop.~2.3]{HairerVar} for a similar argument in the context of Remark \ref{rem:Martin}).
That the renormalisation function $\Upsilon^F[\tau]$ for quadratic $F$ and $\tau\in T_0$ is constant
follows immediately from the definition of
\cite[Eq.~(2.12)]{BCCH}.
Therefore the equation for $\cR\tilde Z^\eps$ is indeed renormalised by a constant, which we then denote by ${\color{black}\hat C_2}$.
With this choice of ${\color{black}\hat C}$ in \eqref{eq:bigsystem}, one has indeed  $\cR(\tilde X^\eps,\tilde Y^\eps,\tilde Z^\eps)=(\hat X^\eps,\hat Y^\eps,\hat Z^\eps)$.
%This will follow from the combinatorial ingredient \cite{BCCH} of the theory of regularity structures.

This finishes the proof, with the caveat that the convergence $Z^\eps=\hat Z^\eps=\cR\tilde Z^\eps\to\cR\tilde Z^0=\bar Z$ obtained from \cite{H0} is local in time. The fact that this convergence also holds globally in time follows from the fact that we a priori know that both $Z^\eps$ and $\bar Z$ belong to $C([0,T],\cC^\theta(\T))$ a.s.\qed

{\color{black}\begin{remark}\label{rem:ini}
It also follows from the theory of regularity structures that $Z^\eps$ depends continuously on $\psi$, uniformly in $\eps$.
%Let us sketch how one would prove Theorem \ref{thm:main} without modifying the initial condition. 
In particular, in the case $\gamma<1/2$ one can simply reverse the modification of the initial condition by considering \eqref{eq:Z} starting from $Z^\eps(0,x)=\psi(x)-\zeta^\eps(x)$ instead, since $\zeta^\eps$ converges in a space of functions with a positive H\"older exponent. In the $\gamma\geq 1/2$ case, $\psi-\zeta^\eps$ is no longer an admissible initial condition for $Z^\eps$. Nevertheless we expect that by enlarging the system of equations with the free solution of the heat equation starting from $\psi-\zeta^\eps$, it is still possible to reverse the modification of the initial condition.
To avoid tedious computations, we do not pursue this.
\end{remark}
}

%%%%%%%%%%%%%%%%%%%%%%%%%%%%%%%%%%%%%%%%%%%%%%%%%%%%%%%%%%%%%%%%%%%
%%                                                               %%
%% Use the two commands below for producing your bibliography    %%
%% with bibtex, then comment again the commands and include the  %%
%% content of the .bbl file in this file below the commands.     %%
%%                                                               %%
%%%%%%%%%%%%%%%%%%%%%%%%%%%%%%%%%%%%%%%%%%%%%%%%%%%%%%%%%%%%%%%%%%%

%\bibliographystyle{amsplain}
%\bibliography{fracKPZ}

% add below the content of your .bbl file produced by bibtex.

%%%%%%%%%%%%%%%%%%%%%%%%%%%%%%%%%%%%%%%%%%%%%%%%%%%%%%%%%%%%%%%%%%%
%%                                                               %%
%% You may add acknowledgments (optional).                       %%
%%                                                               %%
%%%%%%%%%%%%%%%%%%%%%%%%%%%%%%%%%%%%%%%%%%%%%%%%%%%%%%%%%%%%%%%%%%%
\begin{acks}
M.G. is funded by the European Union (ERC, SPDE, 101117125). Views and opinions expressed
are however those of the author(s) only and do not necessarily reflect those of the European Union
or the European Research Council Executive Agency. Neither the European Union nor the granting
authority can be held responsible for them. 
F.T.   was supported by the Austrian Science Fund (FWF): P35428-N.
\end{acks}

%%%%%%%%%%%%%%%%%%%%%%%%%%%%%%%%%%%%%%%%%%%%%%%%%%%%%%%%%%%%%%%%%%%
%%                                                               %%
%% You have reached the end of your document.                    %%
%%                                                               %%
%%%%%%%%%%%%%%%%%%%%%%%%%%%%%%%%%%%%%%%%%%%%%%%%%%%%%%%%%%%%%%%%%%%

\end{document}